\documentclass[12pt]{amsart}
\usepackage{amsfonts, amssymb, amsmath, amsthm, euscript}
\input xy \xyoption{matrix} \xyoption{arrow} 
\def\edge{\ar@{-}}
\def\dropup#1{ \save+<0ex,3ex> \drop{#1} \restore }

\theoremstyle{plain}
\newtheorem{thm}{Theorem}[section]
\newtheorem{lem}[thm]{Lemma}

\newtheorem{cor}[thm]{Corollary}

\theoremstyle{definition}
\newtheorem{setup}[thm]{Setup}
\newtheorem{exam}[thm]{Example}

\newtheorem{note}[thm]{}

\theoremstyle{remark}
\newtheorem*{nonumrem}{Remark}

\def\gq{\gamma_q}
\def\rij{r_{ij}}

\def\nxn{n{\times}n}

\def\Z{\mathbb{Z}}
\def\Zn{\mathbb{Z}^n}
\def\ktimes{k^{\times}}
\def\kplus{k^+}
\def\Ktimes{K^{\times}}
\def\tR{\widetilde{R}}
\def\tL{\widetilde{\Lambda}}
\def\tc{\tilde{c}}
\def\prim{\operatorname{Prim}}
\def\Gplus{\Gamma^+}
\def\kGplus{k\Gamma^+}
\def\kcGplus{k^c\Gamma^+}

\def\kGw{k\Gamma_w}
\def\Gw{\Gamma_w}
\def\KtcGplus{K^{\tilde{c}}\Gamma^+}
\def\im{\operatorname{im}}
\def\Q{\mathbb{Q}}
\def\Z{\mathbb{Z}}
\def\C{\mathbb{C}}
\def\Hom{\operatorname{Hom}}
\def\Spw{S^{\bot}_w} 
\def\spec{\operatorname{Spec}}
\def\pspec{\operatorname{PSpec}}

\def\calC{{\mathcal C}}
\def\calO{{\mathcal O}}
\def\calP{{\mathcal P}}
\def\bfr{{\mathbf r}}
\def\pprim{\operatorname{PPrim}}
\def\sympcore{\operatorname{SympCr}}
\def\chr{\operatorname{char}}
\def\abar{\overline{a}}
\def\bbar{\overline{b}}
\def\gfrak{{\mathfrak g}}
\def\mfrak{{\mathfrak m}}
\def\Max{\operatorname{Max}}
\def\bfq{{\mathbf q}}
\def\qij{q_{ij}}
\def\qji{q_{ji}}
\def\tqij{\tilde{q}_{ij}}
\def\R{\mathbb{R}}
\def\tq{\tilde{q}}
\def\rad{\operatorname{rad}}
\def\Ahat{\widehat{A}}
\def\Rhat{\widehat{R}}
\def\Phihat{\widehat{\Phi}}
\def\chat{\hat{c}}
\def\uhat{\hat{u}}
\def\qhatij{\hat{q}_{ij}}
\def\bfqhat{\widehat{\bfq}}
\def\Max{\operatorname{Max}}
\def\OqG{\calO_q(G)}
\def\gr{\operatorname{gr}}
\def\rhat{\hat{r}}
\def\shat{\hat{s}}
\def\calA{{\mathcal A}}
\def\bfc{{\boldsymbol c}}

\begin{document}

\title{Semiclassical Limits of Quantum Affine Spaces}

\author{K. R. Goodearl}

\author{E. S. Letzter}

\address{Department of Mathematics\\
University of California\\
Santa Barbara, CA 93106}

     \email{goodearl@math.ucsb.edu}

\address{Department of Mathematics\\
        Temple University\\
        Philadelphia, PA 19122}
      
      \email{letzter@temple.edu }

      \thanks{Research of the first author supported by a grant from
        the National Science Foundation, and research of the second
        author by a grant from the National Security Agency.  This
        research was also partially supported by Leverhulme Research
        Interchange Grant F/00158/X (UK)}



\begin{abstract} 
  Semiclassical limits of generic multiparameter quantized coordinate
  rings $A= \calO_{\bfq}(k^n)$ of affine spaces are constructed and
  related to $A$, for $k$ an algebraically closed field of
  characteristic zero and $\bfq$ a multiplicatively antisymmetric
  matrix whose entries generate a torsionfree subgroup of $\ktimes$. A
  semiclassical limit of $A$ is a Poisson algebra structure on the
  corresponding classical coordinate ring $R= \calO(k^n)$, and results
  of Oh, Park, Shin and the authors are used to construct
  homeomorphisms from the Poisson prime and Poisson primitive spectra
  of $R$ onto the prime and primitive spectra of $A$. The Poisson
  primitive spectrum of $R$ is then identified with the space of
  symplectic cores in $k^n$ in the sense of Brown and Gordon, and an
  example is presented (over $\C$) for which the Poisson primitive
  spectrum of $R$ is not homeomorphic to the space of symplectic
  leaves in $k^n$. Finally, these results are extended from quantum
  affine spaces to quantum affine toric varieties. \end{abstract}

\maketitle


\section{Introduction}

This paper is a study of ideal theory in quantum affine
$n$-space and, more generally, in quantum toric varieties. The focus
is on the relationship between the prime and primitive spectra of
these noncommutative algebras and the Poisson spectra of corresponding
commutative semiclassical limits.

\begin{note} {\bf History and context.} A basic principle of the
  \emph{orbit method} is that given a noncommutative algebra $A$, one
  should associate to $A$ an algebraic variety $V$ with a Poisson
  structure and should relate the primitive ideals of $A$ to the
  symplectic leaves in $V$.  This idea first arose in Lie theory, with
  the enveloping algebras $A= U(\gfrak)$ of finite dimensional complex
  Lie algebras $\gfrak$ providing fundamental examples. The symmetric
  algebra $S(\gfrak)$ has a Poisson bracket induced from the Lie
  bracket on $\gfrak$, and the identification of $S(\gfrak)$ with the
  coordinate ring $\calO(\gfrak^*)$ turns the affine space $\gfrak^*$
  into a Poisson variety, equipped with the \emph{KKS Poisson
    structure}. In this setting, the Poisson algebra $S(\gfrak)=
  \calO(\gfrak^*)$ is the \emph{semiclassical limit} of $U(\gfrak)$,
  because its Poisson bracket can be obtained from the identification
  of this algebra with the associated graded algebra of $U(\gfrak)$
  (with respect to the standard filtration). (See \ref{sclim} below
  for the semiclassical limit process.) A famous theorem of Kirillov,
  Kostant, and Souriau shows that the symplectic leaves in $\gfrak^*$
  coincide with the coadjoint orbits of the associated simply
  connected Lie group. If $\gfrak$ is solvable and algebraic, the
  Dixmier map gives a homeomorphism from the space of symplectic
  leaves of $\gfrak^*$ (equipped with the quotient Zariski topology)
  onto the primitive ideal space $\prim U(\gfrak)$.

  Analogous patterns are posited for quantum groups (for which a
  variant of the semiclassical limit is appropriate), particularly for
  quantized coordinate rings of algebraic varieties (e.g.,
  \cite[Introduction]{HodLev1}). One quickly sees, via examples, that
  the best results are to be expected in \emph{generic} cases (meaning
  that appropriate parameters are not roots of unity). Here a
  fundamental test case is $\OqG$, the standard single parameter
  quantized coordinate ring of a complex semisimple algebraic group
  $G$. Hodges and Levasseur \cite{HodLev1},\cite{HodLev2}, working
  with $G= SL_n(\C)$, and Joseph \cite{Jos1},\cite{Jos2}, extending
  their results to general $G$, constructed bijections from the space
  of symplectic leaves in $G$ (relative to the semiclassical limit
  Poisson structure) onto the primitive ideal space of $\OqG$. It is
  an open problem (solved only in the easy case $G= SL_2(\C)$) whether
  homeomorphisms can be constructed. The problem may be alternatively
  expressed in terms of topological quotients. The results of
  Hodges-Levasseur and Joseph give surjections $G\rightarrow \prim
  \OqG$ whose fibers are the symplectic leaves in $G$, and the
  question becomes: Does there exist such a surjection for which
  $\prim \OqG$ has the quotient topology?

  The present authors raised the corresponding problem for other
  quantized coordinate rings in the following form \cite{Goo1},
  \cite{GooLet1}: If $A$ is a generic quantized coordinate ring of an
  algebraic variety $V$, is $\prim A$ a topological quotient of $V$,
  and is the prime spectrum $\spec A$ a topological quotient of $\spec
  \calO(V)$? They proved that these indeed hold for quantum tori and
  quantum affine spaces \cite{GooLet1}. (In fact, these results hold
  in non-generic cases as well, modulo a small technical assumption.)
  Later, Oh, Park, and Shin \cite{OhParShi} showed that the maps
  constructed in \cite{GooLet1} induce homeomorphisms from the spaces
  of Poisson primitive and Poisson prime ideals in coordinate rings of
  tori and affine spaces onto the primitive and prime spectra of
  corresponding generic quantized coordinate rings. However, the Poisson structures in these results were not
exhibited as semiclassical limits.
\end{note}

\begin{note} {\bf Results of this paper.}
  Our purposes here are threefold. First, we construct a semiclassical
  limit $R = \calO(k^n)$ of the quantum coordinate ring $A =
  \calO_q(k^n)$, when the quantizing parameters generate a torsionfree
  group, such that the above-cited results of Oh, Park and Shin can be
  applied to give homeomorphisms, respectively, from the Poisson prime
  and Poisson primitive spectra of $R$ onto the prime and primitive
  spectra of $A$.  Furthermore, since our homeomorphisms occur in
  generic settings, the explicit descriptions of the maps involved can
  be somewhat simplified, and we take the opportunity to do
  so. Second, we show that the Poisson primitive ideals occurring in
  the coordinate rings $\calO(k^n)$ here correspond to the Poisson
  \emph{cores} in the affine space $k^n$, in the sense of Brown and
  Gordon \cite{BroGor}, and do not always correspond to symplectic
  \emph{leaves} in general (over $\C$). Finally, we extend the
  preceding results to quantum affine toric varieties, as we did for
  topological quotients in \cite{Goo1} and \cite{GooLet1}.
\end{note}

\begin{note} The paper is organized as follows. In Section 2, the
  semiclassical limit is constructed. In Section 3, the homeomorphisms
  are presented. Also in Section 3 is an example showing that
  primitive ideals and symplectic leaves do not necessarily correspond
  bijectively. In Section 4, the special case in which the $\qij$ are
  all powers of single parameter is considered. Section 5, finally, contains the
  generalizations to quantum affine toric varieties and related
  algebras.
\end{note}

\begin{note} \label{sclim}
Recall that a \emph{Poisson algebra} over a field $k$ is a commutative $k$-algebra $R$ equipped with a \emph{Poisson bracket}, that is, a $k$-bilinear map $\{-,-\}: R\times R\rightarrow R$ such that $\bigl( R, \{-,-\}\bigr)$ is a Lie algebra and such that $\{-,-\}$ is a derivation in each variable. These can arise as semiclassical limits in the following two ways.

First, suppose that $\calA$ is a nonnegatively filtered $k$-algebra whose associated graded ring, $R= \gr \calA$, is commutative. Given homogeneous elements $r\in \gr_i \calA$ and $s\in \gr_j \calA$, choose representatives $\rhat \in \calA_i$ and $\shat \in \calA_j$. The commutativity of $R$ implies that $[\rhat,\shat] \in \calA_{i+j-1}$, and we set $\{r,s\}$ equal to the coset of $[\rhat,\shat]$ in $\gr_{i+j-1} \calA$. This provides a well-defined Poisson bracket on $R$, and the resulting Poisson algebra is called the \emph{semiclassical limit} of $\calA$.

For the second construction, suppose that $\calA$ is a $k$-algebra and that $h\in \calA$ is a central non-zero-divisor such that $R= \calA/h\calA$ is commutative. Given any $r,s\in R$, choose representatives $\rhat,\shat \in \calA$. Then $[\rhat,\shat]$ is uniquely divisible by $h$ in $\calA$, and we set $\{r,s\}$ equal to the coset $\frac1h [\rhat,\shat] +h\calA$ in $R$. We again obtain a well-defined Poisson bracket on $R$, and this Poisson algebra is viewed as  the \emph{semiclassical limit} of $\calA$. The algebra $\calA$ may be thought of as a family of deformations of $R$, namely the algebras $\calA_q= \calA/(h-q)\calA$ for $q\in k$. By abuse of terminology, $R$ is also referred to as a semiclassical limit of one of the algebras $\calA_q$, when $q$ is suitably generic.
\end{note}

\section{Construction of the Semiclassical Limit}

\begin{setup} \label{setup} Assume throughout the paper that:

(i) $k$ is an algebraically closed
  field of characteristic zero (with group of units $\ktimes$).

(ii) $q \ne 0,1$ is an element of $k$.

(iii) $n$ is a positive integer, and $\bfq= (\qij)$ is a multiplicatively
antisymmetric $\nxn$ matrix over $k$ (i.e., $q_{ii}=1$ and $\qij = \qji^{-1} \in
\ktimes$ for $1 \leq i,j \leq n$).

(iv) The multiplicative subgroup $\langle \qij \rangle = \langle \qij
\mid 1 \leq i,j \leq n \rangle$ of $\ktimes$ is torsion free. (Note
that the rank of this free abelian group can be no larger than
$n(n-1)/2$.)

(v) $A= \calO_{\bfq}(k^n)$ is the $k$-algebra presented by generators
$x_1,\ldots,x_n$ and relations $x_ix_j = \qij x_j x_i$, for $1 \leq
i,j \leq n$. This algebra is commonly referred to as a \emph{multiparameter quantum affine $n$-space over $k$.}

(vi) $K= k[z]_{\langle(z-1)(z-q)\rangle}$ is the localization of a polynomial ring $k[z]$ at the semimaximal ideal $\langle(z-1)(z-q)\rangle$. (If desired, $K$ can be replaced by a finitely generated subalgebra, as noted in (\ref{calculation}).) Write $\Ktimes$ for the group of units of
$K$. Note that there are well-defined evaluation maps
$$\gamma_1: K\rightarrow k \qquad\qquad\text{and}\qquad\qquad \gamma_q: K\rightarrow k,$$
given by $\gamma_1(f)= f(1)$ and $\gamma_q(f)= f(q)$. Moreover, if $f\in\Ktimes$, then $\gamma_1(f),\gamma_q(f) \in \ktimes$.
\end{setup}

Our goal is to realize $\spec A$ via a suitable semiclassical
limit. Our approach depends essentially on \cite{OhParShi}, which in
turn relied on \cite{GooLet2},\cite{GooLet1},\cite{Oh}.

\begin{note} \label{gamma}
(Following \cite[\S 4]{GooLet1} and \cite[\S 1]{GooLet2};
cf.~\cite[\S 4]{BroGoo}.)

  (i) Let $\Gamma = \Zn$, with standard basis elements
  $\epsilon_1,\ldots,\epsilon_n$. For $s = (s_1,\ldots,s_n)$ and $t =
  (t_1,\ldots,t_n)$ in $\Gamma$, set
\[\sigma(s,t) = \prod_{i,j =1}^n \qij^{s_it_j} .\]
Then $\sigma: \Gamma\times\Gamma \rightarrow \ktimes$ is an alternating bicharacter: 
\[ \sigma(s,s) = 1, \qquad \sigma(s,t) = \sigma(t,s)^{-1}, \qquad
\sigma(s,t + u) = \sigma(s,t)\sigma(s,u), \]
for $s,t,u \in \Gamma$. Moreover, the subgroup $\langle \im \sigma
\rangle$ of $\ktimes$ is equal to $\langle \qij \rangle$.

(ii) Let $\Gplus$ denote the submonoid of $\Gamma$ of $n$-tuples
without negative entries. For $s = (s_1,\ldots,s_n)$ and $t =
  (t_1,\ldots,t_n)$ in $\Gplus$, let
$x^s$ denote the monomial
\[ x_1^{s_1}\cdots x_n^{s_n} \in A.\]
Note, for all $s,t \in \Gplus$, that
\[x^sx^t = \sigma(s,t)x^tx^s .\]
Also note that $\qij = \sigma(\epsilon_i,\epsilon_j)$, for $1 \leq i,j,
\leq n$.
\end{note}

\begin{note} \label{existc}
Since $\langle \qij \rangle$ is torsion free, $-1\notin \langle \qij\rangle$. Hence, it follows
  from \cite[Lemma 4.2]{GooLet1} that there exists an alternating
  bicharacter $c : \Gamma \times \Gamma \rightarrow \ktimes$ such that
  $\sigma(s,t) = c(s,t)^2$, for all $s,t \in \Gamma$, and such that
  $\sigma(s,t) = 1$ if and only if $c(s,t) = 1$. In the proof of \cite[Lemma 4.2]{GooLet1}, $c$ is constructed so that the subgroup $\Lambda = \langle \im c \rangle$ of $\ktimes$ is contained in a divisible hull of $\langle \qij \rangle$. Since $\langle \qij\rangle$ is torsionfree, so is its divisible hull, and therefore $\Lambda$ is torsionfree.
 \end{note}

 \begin{note} \label{kcGplus} We now form the twisted monoid algebra
   $\kcGplus$, with $k$-basis $\{x^s \mid s \in \Gplus\}$, and with
   multiplication given via $x^s * x^t = c(s,t)x^{s+t}$, for $s,t \in
   \Gplus$. (The notational overlap with (\ref{gamma}ii) will be
   resolved momentarily.) Note that
\[ x^s * x^t = c(s,t)^2x^t*x^s = \sigma(s,t)x^t*x^s \]
for $s,t \in \Gplus$. Hence, the assignment $x_i \mapsto
x^{\epsilon_i}$, for $1 \leq i \leq n$, induces an isomorphism from
$A$ onto $\kcGplus$. 

Henceforth, we identify $A$ with $\kcGplus$, via the above
isomorphism. In particular, the monomial $x^s$ of (\ref{gamma}ii), for
$s \in \Gplus$, is identified with the basis element $x^s$ of
$\kcGplus$.
\end{note}

\begin{note} \label{lambda} Suppose that $\lambda_1,\ldots,\lambda_m$ form a basis for
  $\Lambda$, with
\[ c(s,t) = \lambda_1^{\ell_1{(s,t)}}\cdots \lambda_m^{\ell_m{(s,t)}} \]
for $s,t \in \Gamma$, and for suitable (unique) alternating biadditive maps $\ell_i: \Gamma\times\Gamma \rightarrow \Z$.
\end{note}

\begin{note} \label{calculation} The field $k$, being algebraically
  closed, must be infinite dimensional over $\Q$. Choose $\Q$-linearly
  independent elements $\mu_1,\dots,\mu_m \in k$. Observe that
  the matrix
$$\left[\begin{matrix} 1&1&1\\ q^2&q&1\\ 2&1&0 \end{matrix}\right] \in M_3(k)$$
has determinant $(q-1)^2 \ne 0$ and so is invertible. Hence, for $1\le
i\le m$, there are unique scalars $a_i,b_i,c_i \in k$ such that the
quadratic polynomial
\[ f_i(z) = a_iz^2 + b_iz + c_i  \]
satisfies the following conditions:
$$f_i(1) = 1, \qquad\qquad f_i(q) = \lambda_i, \qquad\qquad f_i'(1) =
\mu_i \,,$$
where $f'(z)$ denotes the formal derivative of a rational function
$f(z) \in k(z)$. The displayed properties are all that we require of
the polynomials $f_i$. In particular, they need not be quadratic.

Note that $f_1,\dots,f_m \in \Ktimes$, because neither $z-1$ nor $z-q$ is a factor of any $f_i$. Since these are the key properties needed for $K$, we could replace $K$ by the affine algebra $k[z][f_1^{-1},\dots,f_m^{-1}]$, if desired.

Further set
\[ \tc(s,t) =  f_1^{\ell_1{(s,t)}}\cdots f_m^{\ell_m{(s,t)}} \]
for $s,t \in \Gamma$, where $\ell_1,\ldots,\ell_m$ are as in
(\ref{lambda}). Then $\tc : \Gamma \times \Gamma \rightarrow \Ktimes$
is an alternating bicharacter, such that
$$\tc(s,t)(q)=  \lambda_1^{\ell_1{(s,t)}}\cdots \lambda_m^{\ell_m{(s,t)}}= c(s,t)$$
for all $s,t \in \Gamma$.

Set $\tL = \langle \im \tc \rangle= \langle f_1,\dots,f_m\rangle \subset \Ktimes$. \end{note}

\begin{lem} \label{gq} {\rm (i)} The specialization $z \mapsto q$
  induces a group isomorphism
\[ \gq : \tL \xrightarrow{\ f(z) \mapsto f(q)\ } \Lambda .\]

{\rm (ii)} The elements $f_1,\ldots,f_m$ form a basis for $\tL$, and so $\tL$
is free abelian of rank $m$.
\end{lem}

\begin{proof} Consider the group homomorphism
  \[\gq : \Ktimes \xrightarrow{\ f(z) \mapsto f(q)\ } \ktimes .\]
  Note, for $1 \leq i \leq m$, that $\gq(f_i) = \lambda_i$. Since
  $f_1,\ldots,f_m$ generate $\tL$ and $\lambda_1,\ldots,\lambda_m$
  form a basis for $\Lambda$, we see both that the $f_i$ form a basis
  for $\tL$ and that $\gq$ maps $\tL$ isomorphically onto $\Lambda$.
\end{proof}

\begin{nonumrem} By the construction, $\gq(\tc(s,t))
  = c(s,t)$, for all $s,t \in \Gamma$.
\end{nonumrem}

Let $\kplus$ denote the additive group underlying the field $k$.

\begin{lem} \label{psi} The rule $\psi(f)= f'(1)$ gives a well-defined
  injective group homomorphism
\[ \psi : \tL \xrightarrow{\ f(z) \mapsto f'(1)\ } \kplus. \]
\end{lem}

\begin{proof} Observe that $K$ is closed under formal
  differentiation, and so $f'(1)$ is defined for $f\in K$. Thus,
  $\psi$ is a well defined map from $\tL$ to $k$. Now $f_i(1) = 1$ for
 $1 \leq i \leq m$, and so $f(1) = 1$ for all $f \in
  \tL$. Therefore, for all $f,g \in \tL$,
\[ \psi(fg) = f'(1)g(1) + f(1)g'(1) = f'(1) + g'(1)= \psi(f)+ \psi(g), \]
proving that $\psi$ is a group homomorphism. Second, for $1 \leq i \leq m$,
\[ \psi(f_i) = f'_i(1) = \mu_i. \]
Since $\mu_1,\ldots,\mu_m$
are $\Z$-linearly independent, we conclude that $\psi$ is injective.
\end{proof}

\begin{note} \label{Rq}
 Now set $\tR = \KtcGplus$, the twisted monoid
  $K$-algebra with $K$-basis $\{x^s\mid s\in \Gplus\}$ and multiplication given via $x^s * x^t = \tc(s,t)x^{s+t}$, for $s,t \in \Gplus$. We again use $x_i$ to
  denote $x^{\epsilon_i}$, for $1 \leq i \leq n$. 

  Recall that $\qij= \sigma(\epsilon_i,\epsilon_j) = c(\epsilon_i,\epsilon_j)^2$, for $1 \leq i,j
  \leq n$, and set
  \[ \tqij = \tqij(z) = \tc(\epsilon_i,\epsilon_j)^2 \in \Ktimes.\]
Then 
\[ \tR = K\langle x_1,\ldots,x_n \mid x_i*x_j = \tqij x_j*x_i \text{\
  for\ } 1\le i,j\le n \rangle.\]
For $\mu \in \ktimes$, set 
\[ R_\mu = \tR/\langle z - \mu \rangle .\]
We see that $R_q \cong A$, since
$$\tqij(q)= \gamma_q(\tqij)= \gamma_q \tc(\epsilon_i,\epsilon_j)= c(\epsilon_i,\epsilon_j)= \qij$$
 for all $i$, $j$,
and we use this isomorphism to identify $A$ with $R_q$. Under this
identification, the cosets $x_i+ \langle z - q \rangle \in R_q$
correspond to the elements $x_i\in A$. \end{note}

\begin{note} \label{semiclassical_limit}
 Next, consider the group homomorphism
\[ \gamma_1 : \Ktimes \xrightarrow{\ f(z) \mapsto f(1)\ } \ktimes .\]
Note that $\gamma_1(\tL) = 1$ and, in particular, that
$\gamma_1(\tc(s,t)) = 1$ for all $s,t \in \Gamma$. We therefore have
an isomorphism from $R_1$ onto $R := k[x_1,\ldots,x_n]$, sending
\[ x_i + \langle z-1\rangle \quad \longmapsto \quad x_i \,  ,\]
and we identify $R_1$ with $R$ via this map. We also identify $R$ and
$R_1$ with the commutative monoid algebra $\kGplus$, with $k$-basis $\{x^s\mid s\in \Gplus\}$ and multiplication given by $x^sx^t =
x^{s+t}$ for $s,t \in \Gplus$.

Since $\tR/\langle z-1\rangle = R$ is commutative and $z-1$ is a central non-zero-divisor in $\tR$, there is a Poisson bracket on $R$ as in (\ref{sclim}), and $R$ becomes the semiclassical limit of $A = R_q$ (or, more accurately, the
semiclassical limit of the family of algebras $R_\mu$).
The Poisson bracket on $R$ is given by
$$\{\abar,\bbar\}= \left.\dfrac{ab-ba}{z-1}\right|_{z=1}$$
for all $a,b\in \tR$, where $\abar$ and $\bbar$ denote the cosets of $a$ and $b$ in $R$. In particular,
\[ \{x_i,x_j\} = \left.\left(\frac{\tqij(z) -1}{z-1}
  \right|_{z=1}\right)x_ix_j  = \tqij'(1)x_ix_j\]
for $1\le i,j\le n$. (The last equality holds because $\tqij(1)=
\gamma_1(\tqij)= 1$.) 

We treat $R$ as a Poisson algebra in this way, for the remainder of
the paper. In the notation of \cite[3.1]{OhParShi}, $R= k_u\Gplus$,
where $u:\Gamma \times\Gamma \rightarrow k$ is the alternating
biadditive map such that $u(\epsilon_i,\epsilon_j)= \tqij'(1)$, for
$1\le i,j\le n$. In this notation,
$$\{x^s,x^t\}= u(s,t)x^sx^t$$
for $s,t\in \Gplus$.
\end{note}

\section{The Homeomorphisms}

Retain the notation of the previous section. In particular, 
\[ A = R_q = k\langle x_1,\ldots,x_n \mid x_i x_j = \qij x_j
x_i \; \text{for} \; 1 \leq i,j \leq n \rangle  \]
is as in (\ref{Rq}), with $\qij = \tqij(q)$, and
\[ R = R_1 = k[x_1,\ldots,x_n]\]
is the Poisson algebra with bracket
\[ \{ x_i,x_j\} \; = \; \tqij'(1)x_ix_j  \, ,\]
following (\ref{semiclassical_limit}). In particular, $R$ is a
semiclassical limit of $A$. 

\begin{note} \label{pspec} (Following
  \cite{BroGor},\cite{Goo2},\cite{Oh},\cite{OhParShi}.) An ideal $I$
  of $R$ is a \emph{Poisson ideal\/} if $\{R,I\} \subseteq I$. We let
  $\pspec R$ denote the (Zariski) subspace of $\spec R$ comprised of
  the prime Poisson ideals. Each maximal ideal $\mfrak$ of $R$
  contains a unique largest Poisson ideal $\calP(\mfrak)$, called the
  \emph{Poisson core\/} of $\mfrak$. The Poisson cores of the maximal
  ideals of $R$ are termed \emph{Poisson primitive ideals\/}
  \cite[3.2]{BroGor} (or \emph{symplectic ideals\/} \cite[Definition
  1.2]{Oh}) and are prime (\cite[3.2]{BroGor} or \cite[Lemma
  1.3]{Oh}). The subspace of $\pspec R$ consisting of the Poisson
  primitive ideals will be denoted $\pprim R$. The \emph{Poisson
    center\/} $Z_p(R)$ is the set of $z \in R$ such that $\{R,z\} =
  0$.
\end{note}

In our main result, Theorem \ref{mainthm}, we will describe a
homeomorphism from $\pspec R$ onto $\spec A$.

\begin{note} \label{OPSnotation} To proceed further, we need more of
  the notation of \cite{OhParShi}.

  (i) Let $W$ denote the set of subsets of $\{1,\ldots,n\}$, and let
  $w \in W$.

  (ii) Let $I_w$ denote the ideal of $R$ generated by the $x_i$ for $i
  \in w$, and let $Y_w$ denote the multiplicatively closed subset of
  $R/I_w$ generated by $1$ and the cosets of the $x_j$ for $j
  \notin w$. Let $R_w$ denote the localization of $R/I_w$ at $Y_w$. We
  let each $x_i$ also denote its image in $R_w$.

  (iii) Let $\Gw$ denote the subgroup of $\Gamma$ generated by the
  basis elements $\epsilon_j$ for $j \notin w$, and let $c_w$ denote
  the restriction of $c$, defined in (\ref{existc}), to $\Gw \times \Gw$.

  (iv) Identify $R_w$ with the group algebra $\kGw$, via $x_j
  \leftrightarrow x^{\epsilon_j}$, for $j \notin w$.

  (v) Set $H = \Hom(\Gamma, \ktimes)$, which is a group under
  pointwise multiplication, isomorphic to the algebraic
  torus $(\ktimes)^n$. This group acts on $R= \kGplus$ and on $R_w = \kGw$ by $k$-algebra
  automorphisms such that
  \[ h.x^s = \langle h, s \rangle h_s \, ,\]
  for  $h \in H$ and $s\in \Gplus$ or $s \in \Gw$.  Further, set
  \begin{align*}
    S_w &= \rad(c_w)= \{ s \in \Gamma_w \mid c_w(s,t) = 1 \text{\ for
      all\ } t \in
    \Gw \},  \\
    \Spw &= \{ h \in H \mid \langle h, s \rangle = 1\text{\ for all\ }
    s \in S_w \}.
  \end{align*}
Then $\Spw$ is a subgroup of $H$, and it acts on both $R$ and $R_w$ through the $H$-action.

To match the notation of \cite[3.4]{Goo1}, let $\sigma_w$ denote the restriction of $\sigma= c^2$ to $\Gamma_w$. The conditions on $c$ in (\ref{existc}) then imply that
$$S_w= \{ s \in \Gamma_w \mid \sigma_w(s,t) = 1 \text{\ for all\ } t \in
    \Gw \}= \rad(\sigma_w).$$

(vi) Let $\spec_w R$ denote the set of prime ideals of $R$ that
contain $x_i$ for $i \in w$ but do not contain $x_j$ for $j \notin
w$. Localization provides a natural homeomorphism between $\spec_wR$
and $\spec R_w$, and this homeomorphism is $H$-equivariant.

(vii) For $P \in \spec_w R$, let $(P: \Spw)$ denote the intersection of
the prime ideals in the $\Spw$-orbit of $P$.
\end{note}

\begin{note} \label{phi} Both $R$ and $A$ have $k$-basis $\{ x^s \mid
  s\in \Gplus \}$. We will use $\Phi : A \rightarrow R$ to denote the
  $k$-linear isomorphism such that $\Phi(x^s)= x^s$ for all $s\in
  \Gplus$.
\end{note}

\begin{note} \label{pspecw}
(i) For $w \in W$, set
\begin{align*}
 \pspec_w R &= \pspec R \cap \spec _w R \\
  \pprim_w R &= \pprim R \cap\spec_w R .
  \end{align*}

(ii) The maps $\pspec R \rightarrow \spec A$ and $\pprim R \rightarrow \prim A$ used in 
\cite[Theorem 3.5]{OhParShi} are defined by the formula
\[ P\longmapsto \Phi^{-1}(P:\Spw) .\]
for $w \in W$ and $P \in \pspec_wR$. Moreover, as mentioned in the
proof of \cite[Proposition 3.4]{OhParShi}, $(P:\Spw)=P$ for all $P\in
\pspec_w R$ under our present hypotheses, and so the formula reduces to
\[ P\longmapsto \Phi^{-1}(P). \]
Since this key point is somewhat hidden in \cite{OhParShi}, we excerpt
the result and its proof in the next lemma. Before doing so, however,
we need one more ingredient:

(iii) Recall the isomorphism $\gq$ of (\ref{gq}) and the homomorphism $\psi$ of
(\ref{psi}). Let $\varphi$ denote the composite homomorphism
\[\Lambda \xrightarrow{\ \gq^{-1}\ } \tL \xrightarrow{\ f \mapsto f^2\ } \tL
\xrightarrow{\ \psi\ } \kplus , \]
which is injective because $\varphi(\lambda_i)= \psi(f_i^2)= 2\mu_i$ for $1\le i\le n$. Recalling $u$ from
(\ref{semiclassical_limit}), observe that
\[ \varphi(c(\epsilon_i,\epsilon_j)) =
\psi\bigl( \gamma_q^{-1}(c(\epsilon_i,\epsilon_j))^2 \bigr)= \psi\gamma_q^{-1}(\qij)= \psi(\tqij)  = \tqij'(1)= u(\epsilon_i,\epsilon_j) \]
for $1\le i,j\le n$, and so $u=\varphi c$. 
\end{note}

\begin{lem} \label{Spwstable} {\rm \cite{OhParShi}} Let $w\in W$, and
  let $P$ be a Poisson prime ideal in $\spec_w R$. Then $P$ is stable
  under the action of $\Spw$. Consequently, $P = (P : \Spw)$.
\end{lem}

\begin{proof} Let $u_w$ denote the restriction of $u$ to $\Gw$, and
  note that $u_w= \varphi c_w$. The induced Poisson structure on the
  localization $R_w= k\Gw$ satisfies
$$\{x^s,x^t\}= u_w(s,t)x^sx^t$$
for $s,t\in \Gplus_w$, and so we have $R_w= k_{u_w}\Gamma$. 

The Poisson prime ideal $P/I_w$ in $R/I_w$ induces a Poisson prime
ideal $Q$ in $R_w$, which contracts to a prime ideal $Q'$ in the
Poisson center $Z_p(R_w)$. By \cite[Lemma 1.2]{Van} (which is valid
over any base field of characteristic zero), $Q$ is generated by $Q'$,
and $Z_p(R_w)$ equals the group algebra of the radical of
$u_w$. However,
$$ \rad(u_w)= \{s\in \Gw\mid u_w(s,t)=0 \text{\ for all\ } t\in \Gw\}
=\rad(c_w)= S_w\,, $$
because $u_w= \varphi c_w$ and $\varphi$ is injective. Thus, $Z_p(R_w)= kS_w$. By definition of $\Spw$, this group acts trivially on $kS_w$, and so it fixes $Q'$ (pointwise). Therefore $\Spw$ stabilizes $Q$, hence also $P/I_w$, and finally $P$. \end{proof}

We now apply \cite[Theorem 3.5]{OhParShi}. Note that the following theorem asserts, in particular, that the linear map $\Phi^{-1} : R\rightarrow A$ sends Poisson prime ideals of $R$ to prime \emph{ideals} of $A$.

\begin{thm} \label{mainthm} Let $A= \calO_\bfq(k^n)$, where $k$ is an
  algebraically closed field of characteristic zero, and where $\bfq=
  (\qij)$ is a multiplicatively antisymmetric $n\times n$ matrix over
  $k$ such that the group $\langle\qij\rangle \subseteq \ktimes$ is
  torsionfree. Let $R= k[x_1,\dots,x_n]$, equipped with the Poisson
  structure described in {\rm(\ref{semiclassical_limit})}, and let
  $\Phi:A\rightarrow R$ be the $k$-linear isomorphism of
  {\rm(\ref{phi})}. Then the rule $P \mapsto \Phi^{-1}(P)$
  determines a homeomorphism
\[ \pspec R \longrightarrow \spec A, \]
which restricts to a homeomorphism
\[ \pprim R \longrightarrow \prim A. \]
\end{thm} 

\begin{proof} We have a homomorphism $\varphi$, from
  (\ref{pspecw}iii), and an alternating biadditive map $u$, from
  (\ref{semiclassical_limit}), exactly as described in
  \cite[2.3]{OhParShi}. We can identify $R$ as a Poisson $k$-algebra
  with $k_u\Gamma^+$, following the notation of \cite[3.1]{OhParShi};
  see (\ref{semiclassical_limit}). Similarly, we have identified $A$
  as a $k$-algebra with $\kcGplus$, also following the notation of
  \cite[3.1]{OhParShi}; see (\ref{kcGplus}). The theorem now follows
  directly from Lemma \ref{Spwstable} and \cite[Theorem
  3.5]{OhParShi}.
\end{proof}

\begin{note} \label{sympcore}
(i) Identify the affine space $k^n$ with the maximal
  ideal space of $R$. The rule $\mfrak \mapsto \calP(\mfrak)$ gives a
  surjective map $k^n \rightarrow \pprim R$, the fibers of which are
  called \emph{symplectic cores\/} \cite[3.3]{BroGor}, and are
  algebraic analogs of symplectic leaves (cf.~\cite[3.3, 3.5,
  3.7]{BroGor}). Specifically, the symplectic core containing a point
  $\mfrak$ is the set
$$\calC(\mfrak)= \{ \mfrak'\in k^n \mid \calP(\mfrak')= \calP(\mfrak) \}. $$

(ii) Let $\sympcore k^n$ denote the set of symplectic cores in
$k^n$. The rule $\mfrak \mapsto \calC(\mfrak)$ gives a surjective map
$k^n \rightarrow \sympcore k^n$, and we give $\sympcore k^n$ the
quotient (Zariski) topology via this map. By the definition of
symplectic cores, there is a bijection $\sympcore k^n \rightarrow
\pprim R$ such that $\calC(\mfrak) \mapsto \calP(\mfrak)$ for $\mfrak
\in k^n$.

(iii) It follows from \cite[Theorem 4.1(b)]{Goo2} that the Zariski
topology on $\pprim R$ is the quotient topology from the canonical map
$k^n \rightarrow \pprim R$ in (i). (To verify the hypotheses of
\cite[Theorem 4.1(b)]{Goo2}, observe that the action of the
torus $H$ on $R$ is a rational action by Poisson
automorphisms, and observe that $R$ has only finitely many $H$-stable prime
Poisson ideals, namely the ideals $I_w$ of
(\ref{OPSnotation}ii).) Hence, the bijection $\sympcore k^n \rightarrow
\pprim R$ of (ii) is a homeomorphism.
\end{note}

Combining these observations with Theorem \ref{mainthm}, we obtain

\begin{cor} \label{maincor}
Under the hypotheses of Theorem {\rm\ref{mainthm}}, there is a homeomorphism
\[ \sympcore k^n \longrightarrow \prim A \]
given by the rule $\calC(\mfrak) \mapsto \Phi^{-1}(\calP(\mfrak))$. \end{cor} 

The symplectic cores $\calC(\mfrak)$, which make up the points of the
space $\sympcore k^n$ in Corollary \ref{maincor}, have good geometric
structure themselves: they are homogeneous smooth irreducible quasi-affine
varieties, as the following corollary shows. In keeping with the notation
of (\ref{pspecw}i), set
$$\Max_w R= \Max R\cap \spec_w R$$
for $w\in W$. Since we have identified $\Max R$ with $k^n$, the sets
$\Max_w R$ partition $k^n$.

\begin{cor} \label{geomofcores}
Let $w\in W$ and $\mfrak\in \Max_w R$. Then the symplectic core
$\calC(\mfrak)$ is a smooth irreducible locally closed subset of $k^n$,
and it equals the $\Spw$-orbit of $\mfrak$ in $k^n$.
\end{cor}

\begin{proof} Irreducibility and local closedness will follow from
\cite[Lemma 3.3]{BroGor} once we know that all Poisson primitive ideals of
$R$ are locally closed points in $\pspec R$. The latter fact will follow
from the Poisson Dixmier-Moeglin equivalence of \cite[Theorem 4.3]{Goo2}.
As we have already noted in (\ref{sympcore}iii), the torus $H$ acts
rationally on $R$ by Poisson automorphisms, and there are only finitely
many $H$-stable Poisson prime ideals in $R$. Hence, the hypotheses of
\cite[Theorem 4.3]{Goo2} are satisfied. Investing that theorem into
\cite[Lemma 3.3]{BroGor}, we find that $\calC(\mfrak)$ is locally closed
and that its closure  is the set
$$\{\mfrak'\in \Max R \mid \mfrak' \supseteq \calP(\mfrak)\},$$
which is irreducible because $\calP(\mfrak)$ is a prime ideal. Therefore
$\calC(\mfrak)$ is irreducible. \cite[Lemma 3.3]{BroGor} also shows that
$\calC(\mfrak)$ is smooth in its closure. Smoothness in $k^n$ will follow
from the general theory of algebraic group actions (e.g.,
\cite[Proposition 1.8]{Bor}), once we exhibit $\calC(\mfrak)$ as an orbit
of an algebraic group.

The $\Spw$-orbit of $\mfrak$ appears as the fiber of the quotient map
$\Max R \rightarrow \prim A$ in \cite[Theorem 4.11]{GooLet1}. Hence, we
need to show that this quotient map, call it $\mu$, agrees with the one
obtained in our present setting, namely the map
$$\tau: k^n \longrightarrow \prim A, \qquad\qquad \mfrak\longmapsto
\Phi^{-1}(\calP(\mfrak)),$$
which is the composition of the homeomorphism in Corollary \ref{maincor}
with the quotient map $k^n \rightarrow \sympcore k^n$. By construction, the
fibers of $\tau$ are the symplectic cores in $k^n$.

Since $\mfrak\in \Max_w R$, \cite[Proposition 3.4]{OhParShi} shows that
$(\mfrak:\Spw)= (P:\Spw)$ for some $P\in \pspec_w R$. But $(P:\Spw)=P$ by
Lemma \ref{Spwstable}, and so $(\mfrak:\Spw)=P$ is a Poisson primitive
ideal. In particular, $(\mfrak:\Spw)$ is a Poisson ideal contained in
$\mfrak$, whence $(\mfrak:\Spw) \subseteq \calP(\mfrak)$. On the other
hand, the Poisson primitive ideal $\calP(\mfrak)$ is $\Spw$-stable by
Lemma \ref{Spwstable}. Since it is contained in $\mfrak$, it must be
contained in $(\mfrak:\Spw)$. Therefore $\calP(\mfrak)= (\mfrak:\Spw)$,
and we conclude that
$$\tau(\mfrak)= \Phi^{-1}(\mfrak:\Spw).$$
This shows that $\tau$ agrees with $\mu$, as desired. By \cite[Theorem
4.11]{GooLet1}, the fibers of $\mu$ over points in $\prim_w A$ consist
precisely of the $\Spw$-orbits within $k^n$. Therefore $\calC(\mfrak)$
equals the $\Spw$-orbit of $\mfrak$.

Note from its definition that $\Spw$ is a closed subgroup of the torus
$H$, so that it is an affine algebraic group. Since $H$ acts rationally on
$R$, its induced action on $\Max R= k^n$ is morphic, as is the
corresponding action of $\Spw$.  Standard results (e.g., \cite[Proposition
1.8]{Bor}) thus imply that the $\Spw$-orbit $\calC(\mfrak)$ is smooth (and
locally closed).
\end{proof}

When $k=\C$, an affine variety equipped with a Poisson structure can
also be partitioned into symplectic leaves (e.g., see
\cite[3.5]{BroGor}), and it has been a goal of research in quantum
groups to represent primitive spectra of quantized algebras as spaces
of symplectic leaves. This correspondence between
primitive ideals and symplectic leaves, however, can break down when
the symplectic leaves are not algebraic (i.e., not locally closed in
the Zariski topology), as noted by Hodges-Levasseur-Toro
\cite[p.~52]{HLT}, Vancliff \cite[Theorem 3.8]{Van}, and Brandl
\cite[Example 6.4]{Bra}. The following provides an explicit example of
this phenomenon, in the form of a quantum affine $3$-space.

\begin{exam} (i) Let $k=\C$, choose $\alpha\in \R\setminus\Q$, and
  choose $q\in \C$ transcendental over the field $\Q(\alpha)$. We
  would like to construct an example of the form $\calO_\bfq(\C^3)$,
  with some $\qij= q^\alpha$. However, that would require working with
  $z^\alpha$ in our semiclassical limit construction
  (\ref{semiclassical_limit}), and we cannot form $z^\alpha$ in
  $K$. To replace $z^\alpha$, we use the first-order Taylor approximation
  $1+\alpha (z-1)$, and consequently we use $1+\alpha (q-1)$ in place
  of $q^\alpha$ in the defining relations for this example. Because $q$ is transcendental over $\Q(\alpha$), the elements $\lambda_1= q$ and $\lambda_2= 1+\alpha (q-1)$ generate a free abelian subgroup of $\C^\times$ of rank $2$.
  
We now  take
$$\bfq= \left[ \begin{matrix} 1&\lambda_1^2&\lambda_2^2\\ \lambda_1^{-2}&1&1\\ \lambda_2^{-2}&1&1 \end{matrix} \right] $$
and form $A= \calO_\bfq(\C^3)$. 

(ii) The primitive spectrum of $A$ is easily calculated by the methods of \cite{GooLet2}, as follows. Let $W$ denote the set of subsets of $\{1,2,3\}$, and for $w\in W$ set
\begin{align*} J_w &= \langle x_i\mid i\in w\rangle \in \spec A \\
A_w &= (A/J_w)[x_j^{-1}\mid j\notin w] \\
\prim_w A &= \{ P\in \prim A \mid P\cap\{x_1,x_2,x_3\}= \{x_i\mid i\in w\}\,\} \\
S_w &= \{a \in \Z^3 \mid a_i=0 \text{\ for\ } i\in w \text{\ and\ } \prod_{i\notin w} q_{ij}^{a_i}=1 \text{\ for all\ } j\notin w\}.
\end{align*}
Then $\prim A$ is the disjoint union of the sets $\prim_w A$, and each $\prim_w A$ is homeomorphic to $\prim A_w$ via localization and contraction \cite[Theorem 2.3]{GooLet2}. Moreover, since $A_w$ is a quantum torus, $\prim A_w$ consists precisely of the ideals induced from maximal ideals of the center $Z(A_w)$ \cite[Corollary 1.5]{GooLet2}, and $Z(A_w)$ is spanned by the (cosets of the) monomials $x^a$ for $a\in S_w$ (e.g., \cite[Lemma 1.2]{GooLet2}). In particular, when $Z(A_w)=\C$ (equivalently, when $A_w$ is simple), $\prim_w A$ consists of just the ideal $J_w$. In the present example, this occurs in the cases $w=\varnothing, \{2\}, \{3\}, \{1,2,3\}$. Hence,
\begin{align*} 
\prim_{\varnothing} A &= \{\langle0\rangle\}  &\prim_{\{2\}} A &= \{\langle x_2\rangle\} \\
\prim_{\{3\}} A &= \{\langle x_3\rangle\}  &\prim_{\{1,2,3\}} A &= \{\langle x_1,x_2,x_3\rangle\}.
\end{align*}
In the remaining four cases, $A_w$ is a commutative Laurent polynomial ring over $\C$, and so we obtain
\begin{align*}
\prim_{\{1\}} A &= \{ \langle x_1,x_2-b,x_3-c\rangle \mid b,c\in \C^\times\}  \\
\prim_{\{1,2\}} A &= \{ \langle x_1,x_2,x_3-c\rangle \mid c\in \C^\times\}  \\
\prim_{\{1,3\}} A &= \{ \langle x_1,x_2-b,x_3\rangle \mid b\in \C^\times\}  \\
\prim_{\{2,3\}} A &= \{ \langle x_1-a,x_2,x_3\rangle \mid a\in \C^\times\}.
\end{align*}
Hence, $\prim A$ may be pictured as follows:
\medskip
$$\xymatrixrowsep{3pc}\xymatrixcolsep{1.3pc}
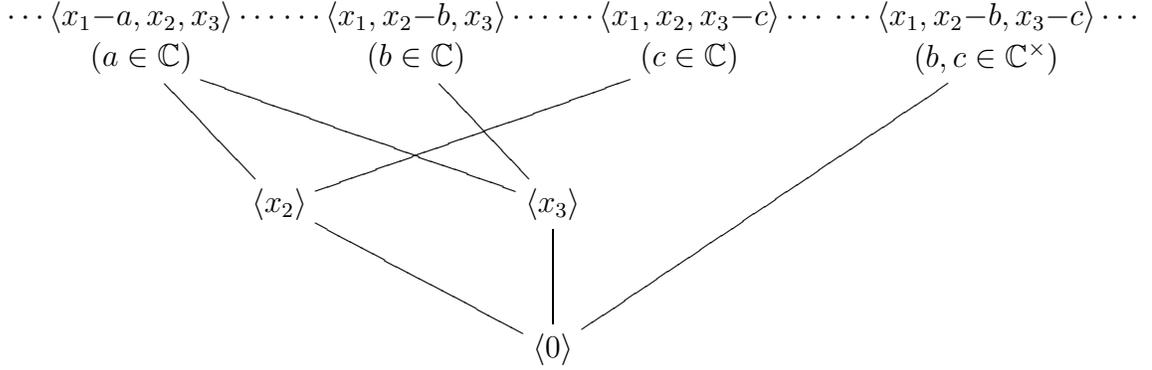
\xymatrix{
(a\in\C) \dropup{\cdots\langle x_1{-}a,x_2,x_3\rangle\cdots} 
 &&(b\in\C) \dropup{\cdots\langle x_1,x_2{-}b,x_3\rangle\cdots}
  &&(c\in\C) \dropup{ \cdots\langle x_1,x_2,x_3{-}c\rangle\cdots}
  &&(b,c\in\C^\times) \dropup{\cdots\langle x_1,x_2{-}b,x_3{-}c\rangle\cdots} \\ 
 &\langle x_2\rangle \edge[ul]\edge[urrr] &&\langle x_3\rangle \edge[ulll]\edge[ul] &&\txt{\hphantom{$\langle x_3\rangle$}} \\ 
 &&&\langle0\rangle \edge[ull]\edge[u]\edge[uurrr]
}$$
\medskip

(iii) In setting up the semiclassical limit, we may choose $c$ as in (\ref{existc}) so that
$$c(\epsilon_1,\epsilon_2)=\lambda_1 \, , \qquad\qquad c(\epsilon_1,\epsilon_3)= \lambda_2 \, , \qquad\qquad c(\epsilon_2,\epsilon_3)= 1.$$
In view of (i), the group $\Lambda= \langle\im c\rangle$ is free abelian with a basis $\lambda_1$, $\lambda_2$. Since $\alpha$ is irrational, we may (and do) choose $\mu_1=1$ and $\mu_2=\alpha$. The polynomials $f_i$ of (\ref{calculation}) are then given by
$$f_1(z)=z, \qquad\qquad f_2(z)= 1+\alpha (z-1),$$
whence $\tq_{12}(z)= z^2$ and $\tq_{13}(z)= (1+\alpha (z-1))^2$, while $\tq_{23}(z)=1$. Consequently,
$$\tq'_{12}(1)= 2, \qquad\qquad \tq'_{13}(1)= 2\alpha, \qquad\qquad \tq'_{23}(1)= 0.$$ 

(iv) The semiclassical limit of $A$ in this example is thus $R= \C[x_1,x_2,x_3]$, equipped with the Poisson structure such that
$$\{x_1,x_2\}= 2x_1x_2, \qquad\quad \{x_1,x_3\}= 2\alpha x_1x_3, \qquad\quad \{x_2,x_3\}= 0.$$
This is a quadratic analog of the KKS Poisson structure on the dual of the standard example of a non-algebraic solvable Lie algebra (eg., cf.~\cite[Example 2.43]{Vanh}).

The Poisson primitive ideals of $R$ can be computed via the Poisson analog of the methods of \cite{GooLet2} -- see \cite[Theorems 4.2, 4.3]{Goo2}. These ideals are given by the same sets of generators as the primitive ideals of $A$ described in (ii); in particular, the picture of $\prim A$ obtained there serves also as a picture of $\pprim R$. From this picture, we read off that the symplectic cores in $\C^3$ for the Poisson structure under consideration are the following sets:
\begin{itemize}
\item The individual points on the $x_1$-axis;
\item The individual points in the $x_2x_3$-plane;
\item The $x_1x_2$-plane with the $x_1$- and $x_2$-axes removed;
\item The $x_1x_3$-plane with the $x_1$- and $x_3$-axes removed;
\item The space $\C^3$ with the three coordinate planes removed.
\end{itemize}

(v) Finally, we indicate how to find the symplectic leaves for this Poisson structure on $\C^3$. These are not all algebraic, just as in the case of the KKS Poisson structure on the dual of a non-algebraic solvable Lie algebra (cf.~\cite[Example 2.43 and discussion on p.~67]{Vanh} and \cite[Remark 1, p.~203]{BroGor}). We first recall from \cite[Proposition 3.6(1)]{BroGor} that each symplectic core is a union of symplectic leaves. This immediately implies that those individual points which are symplectic cores are also symplectic leaves. We determine the other leaves by considering Hamiltonian paths, as follows.

In full, the Poisson bracket on $R$ is given by the formula
$$\{f,g\}= 2x_1x_2 \biggl( \dfrac{\partial f}{\partial x_1} \dfrac{\partial g}{\partial x_2} -\dfrac{\partial f}{\partial x_2} \dfrac{\partial g}{\partial x_1} \biggr) +2\alpha x_1x_3 \biggl( \dfrac{\partial f}{\partial x_1} \dfrac{\partial g}{\partial x_3} -\dfrac{\partial f}{\partial x_3} \dfrac{\partial g}{\partial x_1} \biggr),$$
which also defines the unique extension of the bracket to the algebra $S$ of smooth complex functions on $\C^3$. For $f\in S$, the derivation $H_f= \{f,-\}$ on $S$ gives a smooth vector field on $\C^3$, and the flows (integral curves) of such vector fields $H_f$ are the \emph{Hamiltonian paths} for the given Poisson structure. (Specifically, a flow of $H_f$ is a path $\bfc(t)= (x_1(t),x_2(t),x_3(t))$ such that
$$\dot\bfc(t)= H_f \bfc(t)= \biggl( -2x_1x_2\dfrac{\partial f}{\partial x_2} -2\alpha x_1x_3 \dfrac{\partial f}{\partial x_3} ,\ 2x_1x_2 \dfrac{\partial f}{\partial x_1},\ 2\alpha x_1x_3 \dfrac{\partial f}{\partial x_1} \biggr),$$
where the dot stands for $d/dt$.) By definition \cite[p.~529]{Wei}, the symplectic leaves of $\C^3$ are the equivalence classes for the relation `connected by piecewise Hamiltonian paths'.

Paths of the form $\bfc(t)= (a,be^{at},ce^{\alpha at})$ with $a,b,c\in\C$ are flows of $H_{x_1/2}$, and paths $\bfc(t)= (ae^{bt},b,c)$ with $a,b,c\in\C$ are flows of $H_{-x_2/2}$. It follows that the third and fourth of the symplectic cores listed in (iv) are connected with respect to piecewise Hamiltonian paths, and hence these cores are also symplectic leaves. Moreover, each of the surfaces
$$\Sigma_d= \{ (a_1,a_2,a_3) \in (\C^\times)^3 \mid a_3= da_2^\alpha \},$$
for $d\in\C^\times$, is connected with respect to piecewise Hamiltonian paths. On the other hand, any Hamiltonian path within $(\C^\times)^3$ must satisfy $\dot x_3(t)/x_3(t)= \alpha\dot x_2(t)/x_2(t)$ and so must be contained within one of the $\Sigma_d$. Therefore the $\Sigma_d$ are the remaining symplectic leaves of $\C^3$. These form a one-parameter family of non-algebraic surfaces whose union is a single symplectic core.

We conclude that while $\prim A$ is homeomorphic to $\sympcore \C^3$ (Corollary \ref{maincor}), it is not homeomorphic to the space of symplectic leaves in $\C^3$. For instance, $\prim A$ has just one dense point, while the space of symplectic leaves in $\C^3$ (equipped with the quotient Zariski topology) has uncountably many.
\end{exam}

\section{Uniparameter Quantum Affine Spaces}

The Poisson structure on $k^n$ in Theorem \ref{mainthm} can be given more explicitly in the uniparameter case, namely when the scalars $\qij$ are powers of a single scalar which is not a root of unity. We treat this scalar as a square (as we may, since $k$ is algebraically closed), and write it in the form $q^2$ to match our existing notation.

\begin{note} \label{uniparameter} (i) Assume that the scalar
  $q\in\ktimes$ is not a root of unity. Let $\bfr= (\rij)$ be an
  additively antisymmetric $n\times n$ matrix over $\Z$, and take
  $\qij= q^{2\rij}$ for all $i$, $j$. Assume that $\bfr\ne \mathbf 0$,
  so that at least one $\qij\ne 1$.

  (ii) Under the present assumptions, the natural choice for the
  multiplicatively antisymmetric bicharacter $c$, as in
  (\ref{existc}), is to define it so that
$$c(\epsilon_i,\epsilon_j)= q^{\rij}$$
for all $i$, $j$. The group $\Lambda= \langle \im c\rangle$ is then
cyclic, of the form $\langle q^r\rangle$ where $r$ is the greatest
common divisor of the integers $\rij$.  Note that $r\ne 0$ because
$\bfr\ne \mathbf 0$.

(iii) Take $\lambda_1= q^r$ as basis element for $\Lambda$, and let $\ell_1$ be the corresponding alternating biadditive map on $\Gamma$ as in (\ref{lambda}), so that $c(s,t)= q^{r\ell_1(s,t)}$ for $s,t\in \Gamma$. Thus, $\ell_1(\epsilon_i,\epsilon_j)= \rij/r$ for $1\le i,j\le n$. 

(iv) Since $r\ne 0$, we may take $\mu_1=r$. It is most convenient to take a possibly non-\linebreak[0]quadratic choice for the polynomial $f_1$ in (\ref{calculation}), namely $f_1(z)= z^r$. It is easily seen that this satisfies the desired conditions:
$$f_1(1) = 1, \qquad\qquad  f_1(q) = \lambda_1,  \qquad\qquad f_1'(1) = \mu_1 \,.$$
Define $\tc$ as in (\ref{calculation}), and observe that
$$\tc(\epsilon_i,\epsilon_j)= z^{r\ell_1(\epsilon_i,\epsilon_j)}= z^{\rij}$$
 for $1\le i,j\le n$. 
 
 (v) Defining $\tqij$ as in (\ref{Rq}), we have $\tqij(z)= \tc(\epsilon_i,\epsilon_j)^2= z^{2\rij}$ for $1\le i,j\le n$. Consequently,
$$\tqij'(1)= 2\rij$$
for $1\le i,j\le n$.
\end{note}

In view of (\ref{uniparameter}), the uniparameter cases of Theorem \ref{mainthm} and Corollary \ref{maincor} can be stated as follows:

\begin{thm} \label{uniparamthm} Let $A= \calO_\bfq(k^n)$, assuming
  that $k$ is an algebraically closed field of characteristic zero and  $\bfq= (q^{2\rij})$,
where $q\in\ktimes$ is not a root of unity and $(\rij)$ is a nonzero additively antisymmetric $n\times n$ matrix
  over $\Z$.  Let $R= k[x_1,\dots,x_n]$, equipped with the Poisson
  structure such that
  $$\{x_i,x_j\}= 2\rij x_ix_j$$
  for $1\le i,j\le n$, and let
  $\Phi:A\rightarrow R$ be the $k$-linear isomorphism of
  {\rm(\ref{phi})}. Then the rule $P \mapsto \Phi^{-1}(P)$
  determines a homeomorphism
\[ \pspec R \longrightarrow \spec A, \]
that restricts to a homeomorphism
\[ \pprim R \longrightarrow \prim A. \]
Moreover, the rule $\calC(\mfrak) \mapsto \Phi^{-1}(\calP(\mfrak))$ determines a homeomorphism
\[ \sympcore k^n \longrightarrow \prim A. \]
\end{thm}

\section{Quantum Affine Toric Varieties}

We extend our main results to quantizations of affine toric varieties
and, somewhat more generally, to certain cocycle twists of affine
commutative algebras. By a \emph{quantum affine toric variety} over
$k$ we mean, as in \cite{Ing}, an affine domain over $k$ equipped with
a rational action of an algebraic torus $H$ by $k$-algebra
automorphisms, such that the $H$-eigenspaces are
$1$-dimensional. Since a rational action of $H$ is equivalent to a
grading by the character group of $H$ (e.g., \cite[Lemma
II.2.11]{BroGoo}), the quantum affine toric varieties over $k$ are
also the affine domains over $k$, graded by free abelian groups of
finite rank, with $1$-dimensional homogeneous components. This second
description is convenient for present purposes, as it allows us to
define quantizations via cocycle twists. As in \cite[Section
6]{GooLet1} and \cite[Section 4]{Goo1}, neither $1$-dimensionality of
homogeneous components nor absence of zero-divisors is needed in our
proofs, and so we can work with a more general class of twists of
graded algebras.

\begin{note} \label{twist} (i) Let $R$ be a commutative affine
  $k$-algebra graded by an abelian group $G$, and let $c: G\times
  G\rightarrow \ktimes$ be a $2$-cocycle. The \emph{twist of $R$ by
    $c$} \cite[Section 3]{ArtSchTat} is a $G$-graded $k$-algebra $R'$,
  with a $G$-graded vector space isomorphism $r\mapsto r'$ from
  $R\rightarrow R'$ (the \emph{twist map}), and multiplication given
  by $r's'= c(\alpha,\beta)(rs)'$ for $\alpha,\beta\in G$ and $r\in
  R_\alpha$, $s\in R_\beta$.

In \cite[Theorem 6.3]{GooLet1}, topological quotient maps $\spec R\rightarrow \spec R'$ and $\max R\rightarrow \prim R'$ are constructed, under the assumptions that $G$ is torsionfree and $-1\notin \langle\im c\rangle$ (or $\chr k =2$). As discussed in \cite[4.3]{Goo1}, the proof of \cite[Theorem 6.3]{GooLet1} provides an alternating bicharacter $d$ on $G$ such that $R'$ is isomorphic to the twist of $R$ by $d$, and after replacing $c$ by $d$, the torsionfreeness hypothesis on $G$ is no longer needed.

Thus, we now assume that $c$ is an alternating bicharacter on $G$. Moreover, we assume that the subgroup $\langle\im c\rangle \subseteq \ktimes$ is torsionfree. Set $A=R'$, and let $\Phi: A\rightarrow R$ be the inverse of the twist map.

(ii) Since $R$ is affine, we can choose a finite set of homogeneous $k$-algebra generators for $R$, say $r_1,\dots,r_n$. Set $\delta_i= \deg r_i$ for $1\le i\le n$. Define $\Gamma$ and $\Gplus$ as in (\ref{gamma}), let $\rho: \Gamma\rightarrow G$ be the group homomorphism such that $\rho(\epsilon_i)= \delta_i$ for $1\le i\le n$, and set $\chat= c\circ(\rho\times\rho)$, which is an alternating bicharacter on $\Gamma$. Also, set $\qhatij= \chat(\epsilon_i,\epsilon_j)= c(\delta_i,\delta_j)$ for $1\le i,j\le n$ and $\bfqhat= (\qhatij)$.

Now set $\Ahat= \calO_{\bfqhat}(k^n)$ and identify $\Ahat$ with
$k_{\chat}\Gplus$, with $k$-basis $\{x^s \mid s\in \Gplus\}$. The
corresponding semiclassical limit, as in (\ref{semiclassical_limit}),
is the Poisson algebra $\Rhat= k_{\uhat}\Gplus$, for a suitable
alternating biadditive map $\uhat= \varphi\chat: \Gamma\times
\Gamma\rightarrow k$, where $\varphi$ is an injective group
homomorphism from $\langle\qhatij\rangle= \langle\im \chat\rangle$ to
$k^+$, as in (\ref{pspecw}iii). We also write $\Rhat$ with $k$-basis
$\{x^s \mid s\in \Gplus\}$. Hence, there is a $k$-linear isomorphism
$\Phihat: \Ahat\rightarrow \Rhat$ such that $\Phihat(x^s)= x^s$ for
all $s\in\Gplus$, as in (\ref{phi}).

(iii) Let $\pi_A: \Ahat\rightarrow A$ and $\pi_R: \Rhat\rightarrow R$
be the natural $k$-algebra quotient maps, such that $\pi_A(x_i)= r'_i$
and $\pi_R(x_i)= r_i$ for $1\le i\le n$. Then we obtain a
diagram of $k$-linear maps as follows:
$$\xymatrixrowsep{3pc}\xymatrixcolsep{5pc}
\xymatrix{
{\Rhat} \ar@{->>}[d]_{\pi_R} &{\Ahat} \ar[l]_{\Phihat}
\ar@{->>}[d]^{\pi_A}\\ 
R &A \ar[l]_{\Phi}
}$$
This diagram commutes because
\begin{align*} \Phi\pi_A(x^s) &= \biggl( \prod_{1\le i<j\le n} \tc(s_i\epsilon_i, s_j\epsilon_j) \biggr)^{-1} \Phi\pi_A(x_1^{s_1} \cdots x_n^{s_n}) \\
 &= \biggl( \prod_{1\le i<j\le n} c(s_i\delta_i, s_j\delta_j) \biggr)^{-1} \Phi\bigl( (r_1^{s_1})' \cdots (r_n^{s_n})' \bigr) \\
  &= \Phi\bigl( (r_1^{s_1}\cdots r_n^{s_n})' \bigr)= r_1^{s_1}\cdots r_n^{s_n} = \pi_R\Phihat(x^s)
  \end{align*}
for $s\in \Gplus$.

The $\Gamma$-grading on $\Rhat$ induces a $G$-grading via the homomorphism $\rho$, which we write in the form $\Rhat= \bigoplus_{\alpha\in G} \Rhat[\alpha]$, where
$$\Rhat[\alpha]= \bigoplus_{s\in \rho^{-1}(\alpha) \cap \Gplus} kx^s$$
for $\alpha\in G$. With respect to this $G$-grading, $\pi_R$ is $G$-homogeneous, in the sense that $\pi_R(\Rhat[\alpha]) \subseteq R_\alpha$ for all $\alpha\in G$. Hence, $\ker\pi_R$ is a $G$-homogeneous ideal of $\Rhat$.

(iv) We next show that $\ker\pi_R$ is a Poisson ideal of $\Rhat$. To see this, let $\alpha,\beta\in G$ and note that whenever $s\in \rho^{-1}(\alpha) \cap \Gplus$ and $t\in \rho^{-1}(\beta) \cap \Gplus$, we have
$$\{x^s,x^t\}= \varphi\chat(s,t)x^sx^t= \varphi c(\alpha,\beta)x^sx^t.$$
It follows that $\{a,b\}= \varphi c(\alpha,\beta)ab$ for all $a\in \Rhat[\alpha]$ and $b\in \Rhat[\beta]$. Consequently, any $G$-homogeneous ideal of $\Rhat$, and in particular $\ker\pi_R$, is a Poisson ideal.

Now $R$ becomes a Poisson algebra quotient of $\Rhat$, such that 
$$\{a,b\}= \varphi c(\alpha,\beta)ab$$
whenever $\alpha,\beta\in G$ and $a\in R_\alpha$, $b\in R_\beta$. (In particular, $\{r_i,r_j\}= \varphi c(\delta_i,\delta_j)r_ir_j$ for $1\le i,j\le n$.) We view $R$, equipped with this Poisson structure, as a semiclassical limit of $A$.
\end{note}

\begin{thm} \label{qtorichomeos}
Let $k$ be an algebraically closed field of characteriztic zero and $R$ a commutative affine $k$-algebra, graded by an abelian group $G$. Let $c:G\times G\rightarrow \ktimes$ be an alternating bicharacter such that the group $\langle\im c\rangle \subseteq \ktimes$ is torsionfree, let $A$ be the twist of $R$ by $c$, and let $\Phi: A\rightarrow R$ be the inverse of the twist map. Equip $R$ with the Poisson structure described in {\rm(\ref{twist}iv)}. Then the rule $P \mapsto \Phi^{-1}(P)$ determines a homeomorphism
\[ \pspec R \longrightarrow \spec A, \]
which restricts to a homeomorphism
\[ \pprim R \longrightarrow \prim A. \]
Moreover, the rule $\calC(\mfrak) \mapsto \Phi^{-1}(\calP(\mfrak))$ determines a homeomorphism
\[ \sympcore \Max R \longrightarrow \prim A. \]
\end{thm}

\begin{proof} By Theorem \ref{mainthm}, the rule $P\mapsto \Phihat^{-1}(P)$
determines homeomorphisms
\[ \pspec \Rhat \longrightarrow \spec \Ahat \qquad\qquad \text{and} \qquad\qquad  \pprim \Rhat \longrightarrow \prim \Ahat. \]
Observe that the first homeomorphism restricts to a homeomorphism $\eta: V\rightarrow W$ where
$$V= \{P\in \pspec\Rhat \mid P\supseteq \ker\pi_R\} \qquad \text{and} \qquad W= \{P\in \spec\Ahat \mid P\supseteq \ker\pi_A\}.$$
The quotient maps $\pi_R$ and $\pi_A$ induce homeomorphisms $\pi_R^*: \pspec R \rightarrow V$ and $\pi_A^*: \spec A \rightarrow W$, which fit into the following commutative diagram:
 $$\xymatrixrowsep{3pc}\xymatrixcolsep{6pc}
\xymatrix{
 {\pspec\Rhat} \ar[r]^{P\longmapsto \Phihat^{-1}(P)} &{\spec\Ahat}\\
 V \ar[u]^{\subseteq} \ar[r]^{\eta} &W \ar[u]_{\subseteq}\\
 {\pspec R} \ar[u]^{\pi_R^*} \ar[r]^{P\longmapsto \Phi^{-1}(P)} &{\spec A} \ar[u]_{\pi_A^*}
 }$$
Thus, we have the desired homeomorphism $\pspec R \longrightarrow \spec A$.
 
 That $P\longmapsto \Phi^{-1}(P)$ also determines a homeomorphism $\pprim R \longrightarrow \prim A$ follows in the same manner, once one observes that $\pi_R^*$ and $\pi_A^*$ map $\pprim R$ and $\prim A$ homeomorphically onto $V\cap \pprim\Rhat$ and $W\cap \prim\Ahat$, respectively.
 
 The final homeomorphism will follow from the results above in the same manner as Corollary \ref{maincor} once we show that the Zariski topology on $\pprim R$ is the quotient topology induced by the Poisson core map $\calP(-): \Max R\rightarrow \pprim R$. Set
$$X= \{\mfrak\in \Max \Rhat\mid \mfrak\supseteq \ker\pi_R\},$$
and observe that we have a commutative diagram
$$\xymatrixrowsep{3pc}\xymatrixcolsep{4pc}
\xymatrix{
{\Max\Rhat} \ar[r]^-{\calP(-)} &{\pprim\Rhat}\\
X \ar[u]^{\subseteq} \ar[r]^-{\theta} &{V\cap \pprim\Rhat} \ar[u]_{\subseteq}\\
{\Max R} \ar[u]^{\pi_R^*} \ar[r]^-{\calP(-)} &{\pprim R} \ar[u]_{\pi_R^*}
}$$
with surjective horizontal maps.  As in (\ref{sympcore}iii), it
follows from \cite[Theorem 4.1(b)]{Goo2} that the topology on $\pprim
\Rhat$ is the quotient topology from the top map in the diagram. It
follows that $\theta$ is a topological quotient map, and therefore so
is the bottom map, as desired.
 \end{proof}
 
 \begin{note} The uniparameter case of Theorem \ref{qtorichomeos} is the case in which $c= q^d$ where $q\in\ktimes$ is not a root of unity and $d:G\times G \rightarrow k$ is an antisymmetric biadditive map. We can then take $\varphi: \langle\im c\rangle \rightarrow k$ to be the $q$-logarithm, so that $\varphi c(\alpha,\beta)= d(\alpha,\beta)$ for $\alpha,\beta\in G$. The Poisson structure on $R$ is then given by
$$\{a,b\}= d(\alpha,\beta)ab$$
for $\alpha,\beta\in G$ and $a\in R_\alpha$, $b\in R_\beta$.
\end{note}


\end{document}